\documentclass[12pt,oneside,reqno]{amsart}
\usepackage[utf8]{inputenc}
\usepackage[english]{babel}

\usepackage[margin=1in]{geometry}

\usepackage{multicol}

\usepackage{amsmath}
\usepackage{amsthm}
\usepackage{amssymb}
\usepackage{graphicx}
\usepackage{mathrsfs}
\usepackage[colorinlistoftodos]{todonotes}
\usepackage[colorlinks=true, allcolors=black]{hyperref}
\usepackage{comment}
\usetikzlibrary{graphs}
\usepackage{float}
\usepackage{xcolor}
\usepackage{tikz-cd}
\usepackage{manfnt}

\newcommand{\showcomments}{yes}

\newsavebox{\commentbox}
{\ifthenelse{\equal{\showcomments}{yes}}%
{\footnotemark
        \begin{lrbox}{\commentbox}
        \begin{minipage}[t]{1in}\raggedright\sffamily\tiny
        \footnotemark[\arabic{footnote}]}
{\begin{lrbox}{\commentbox}}}%
{\ifthenelse{\equal{\showcomments}{yes}}%
{\end{minipage}\end{lrbox}\marginpar{\usebox{\commentbox}}}
{\end{lrbox}}}

\makeatletter
\newcommand{\doublewidetilde}[1]{{%
  \mathpalette\double@widetilde{#1}%
}}
\newcommand{\double@widetilde}[2]{%
  \sbox\z@{$\m@th#1\widetilde{#2}$}%
  \ht\z@=.9\ht\z@
  \widetilde{\box\z@}%
}
\makeatother

\usepackage{tikz}
   \usetikzlibrary{calc}

\newcommand{\overbow}[1]{
   \tikz [baseline = (N.base), every node/.style={}] {
      \node [inner sep = 0pt] (N) {$#1$};
      \draw [line width = 0.4pt] plot [smooth, tension=1.3] coordinates {
         ($(N.north west) + (0.1ex,0)$)
         ($(N.north)      + (0,0.5ex)$)
         ($(N.north east) + (0,0)$)
      };
   }
}

\newtheorem{thm}{Theorem}[section]

\newtheorem{theorem}[thm]{Theorem}
\newtheorem{corollary}[thm]{Corollary}
\newtheorem{lemma}[thm]{Lemma}
\newtheorem{proposition}[thm]{Proposition}

\newtheorem*{theorem*}{Theorem}

\theoremstyle{definition}

\newtheorem{construction}[thm]{Construction} 
\newtheorem{definition}[thm]{Definition}

\theoremstyle{remark}

\newtheorem{conv}[thm]{Convention}

\newtheorem{remark}[thm]{Remark}

\newtheorem{example}[thm]{Example}

\newcommand{\nclose}[1]{\ensuremath{\langle\!\langle#1\rangle\!\rangle}}

\newcommand{\scname}[1]{\text{\sf #1}}
\newcommand{\area}{\scname{Area}}

\newcommand{\field}[1]{\mathbb{#1}}

\newcommand{\rationals}{\ensuremath{\field{Q}}}
\newcommand{\naturals}{\ensuremath{\field{N}}}

\title{Asphericity of cubical presentations: the 2-dimensional case}

\author{Macarena Arenas}
\address{DPMMS, Centre for Mathematical Sciences, Wilberforce Road, Cambridge, CB3 0WB, UK}
\email{mcr59@dpmms.cam.ac.uk}

\thanks{The author was supported by a Cambridge Trust \& Newnham College Scholarship, and by the Denman Baynes Junior Research Fellowship at Clare College, Cambridge.}
  
\begin{document}

\begin{abstract}
We show that under suitable hypotheses, the second homotopy group of the coned-off space associated to a $C(9)$ cubical presentation is trivial, and use this to provide classifying spaces for proper actions for the fundamental groups of many quotients of square complexes admitting such cubical presentations. When the cubical presentations satisfy a condition analogous to requiring that the relators in a group presentation are not proper powers, we conclude that the corresponding coned-off space is aspherical.
\end{abstract}

\maketitle

\section{Introduction}

The aim of this work is to explore the asphericity of certain cubical presentations $X^*=\langle X | \{Y_i \rightarrow X\} \rangle$ where $X$ is a compact, non-positively curved cube complex and $Y_i \rightarrow X$ are local isometries of compact non-positively curved cube complexes.
A finite cubical presentation is a natural generalisation of a finite group presentation $\mathcal{P}=\langle s_1, \ldots\, s_k| r_i, \ldots, r_\ell \rangle$.

It is a well-known result of Lyndon~\cite{Lyn66} that classical $C(6)$ presentations are aspherical if no relators are proper powers, and thus that the groups admitting such presentations have cohomological dimension at most equal to $2$. This was generalised to the setting of graphical $C'(\frac{1}{6})$ small-cancellation presentations by Gromov~\cite{Gromov2003} and Ollivier~\cite{Ollivier06}, and graphical $C(6)$ presentations by Gruber~\cite{Gruber15}, and to the setting of small-cancellation for rotation families of groups by Coulon~\cite{Cou11}. 
In what follows, we take the first steps towards a generalisation of these results to the setting of cubical small-cancellation theory, and show:

\begin{theorem}\label{thm:introA}
Let $X^*=\langle X | \{Y_i\}_{i \in I} \rangle$ be a minimal cubical presentation that satisfies the $C(9)$ condition. Let $\pi_1 X^*=\pi_1X/\langle \langle \{\pi_1Y_i\}_{i \in I}\rangle \rangle=:G$. If $dim(X)\leq 2$ and each $Y_i$ is homotopy equivalent to a graph, then $X^*$ is a $K(G,1)$, so $G$ is torsion-free and $gd(G)\leq 2$.
\end{theorem}

The above hypothesised minimality is the cubical analogue of requiring, in the classical setting, that none of the relators are proper powers, and is necessary to avoid torsion. See Definition~\ref{def:minimal} below.

Replacing minimality with the weaker hypothesis that $X^*=\langle X | \{Y_i\}_{i \in I} \rangle$ be symmetric in the sense of Definition~\ref{def:symmetric}, we obtain instead classifying spaces for proper actions for $\pi_1 X^*$. These spaces arise in connection to the Baum-Connes conjecture, and thus this is part of the motivation for finding good models for $\underbar EG$.

\begin{theorem}\label{thm:introB}
Let $X^*=\langle X | \{Y_i\}^k_{i=1} \rangle$ be a symmetric cubical presentation that satisfies the $C(9)$ condition. Let $\pi_1 X^*=\pi_1X/\langle \langle \{\pi_1Y_i\}^k_{i=1}\rangle \rangle=:G$. If  $dim(X)\leq 2$ and each $Y_i$ is homotopy equivalent to a graph, then there is a quotient $\bar X^*$ of $\widetilde{X^*}$ that is an $\underbar EG$, so $cd_\rationals (G) \leq dim(\bar X^*) \leq 2$.  If, in addition, $X$ has a finite regular cover where each $Y_i \rightarrow X$ lifts to an embedding, then $vcd(G) \leq 2$.
\end{theorem}

Since some of the cubically presented groups satisfying the hypotheses above have torsion, Theorem~\ref{thm:introB} is the best statement possible in this generality. 

Theorems~\ref{thm:introA} and~\ref{thm:introB} will follow from Theorem~\ref{thm:cub2}, which applies to cubical presentations of arbitrary dimension and is of independent interest.

Assuming minimality or symmetry, and under the hypotheses on the dimension of $X$ and the $Y_i$'s stated above, Theorems~\ref{thm:introA} and~\ref{thm:introB} answer a question posed by Wise in~\cite[4.5]{WiseIsraelHierarchy}. 
The general case of the theorem, where the dimensions of $X$ and  $Y_i$ are arbitrary, and which we use to derive near-sharp bounds on the (virtual) cohomological dimension of $X^*$,  will be treated in forthcoming work~\cite{Arenas2023pi}.

Classical and graphical small-cancellation groups are hyperbolic as soon as they satisfy either the  $C'(\frac{1}{6})$ or the $C(7)$ conditions, and the groups produced in \cite{Cou11} are also hyperbolic. Thus, our main theorem is particularly interesting in that hyperbolicity is neither assumed nor deduced: while one can conclude hyperbolicity for a sufficiently good cubical small cancellation presentation if one assumes hyperbolicity of $\pi_1X$, it is nevertheless the case that non-hyperbolic, and even non-relatively-hyperbolic groups can satisfy strong cubical small-cancellation conditions.
Thus, Theorems~\ref{thm:introA} and~\ref{thm:introB} apply to a large and varied family of quotients of cubulated groups. 

Despite their indisputable value as a source of examples of interesting group-theoretic behaviour, the classical and graphical small-cancellation theories are limited in their applicability in that they can only serve to understand ``nice'' quotients of free groups; cubical small-cancellation theory transfers some of the complexity arising from the relators in a group presentation to the ``cubical generator'' in the cubical presentation -- that is, the cube complex $X$ that is being quotiented -- making it possible to exploit the tools of cubical geometry to prove more general theorems. 

This approach has already proven fruitful in many instances: particularly, it plays a key role in Agol's celebrated proofs of the Virtual Haken and Virtual Fibred Conjectures~\cite{AgolGrovesManning2012, Agol08}, which build on work of Wise~\cite{WiseIsraelHierarchy} and his collaborators~\cite{BergeronWiseBoundary, HaglundWiseAmalgams, HsuWiseCubulatingMalnormal}; it is used by Arzhantseva and Hagen in~\cite{ArHag2021} to show that many groups that arise as quotients of cubulated groups are acylindrically hyperbolic; it is used by Jankiewicz and Wise in~\cite{jankiewicz2017cubulating} to construct fundamental groups of compact nonpositively curved cube complexes that do not virtually split, and it is used by the author in~\cite{Arenas2023} to produce a version of the Rips construction that provides cocompactly cubulated hyperbolic groups of arbitrarily large finite cohomological dimension, many of which algebraically fibre.

\subsection{Structure of the paper}
In Section~\ref{sec:asphericitycube}, as a sort of base-case for our main theorem, we reprove without resorting to CAT(0) geometry, the well-known result that CAT(0) cube complexes are contractible.
In Section~\ref{sec:back} we  present the necessary background on cubical small cancellation theory, diagrams in cell complexes, and asphericity. 
 In Section~\ref{sec:main} we give the proof of our main technical result, Theorem~\ref{thm:cub2}, and use it to deduce Theorems~\ref{thm:introA} and~\ref{thm:introB}. 
 Finally, in Section~\ref{sec:ex} we outline some examples to which our theorem applies.

\subsection*{Acknowledgements} I am grateful to my PhD supervisor, Henry Wilton, and to Jack Button, Mark Hagen, and the anonymous referee  
 for all their comments and suggestions, which greatly improved the quality of this work.

\section{On the asphericity of non-positively curved cube complexes}\label{sec:asphericitycube}

We will assume that the reader is familiar with the basic background on cube complexes. Particularly with the definitions of \emph{non-positively curved cube complex}, \emph{local isometry}, \emph{midcube}, \emph{hyperplane}, and \emph{hyperplane carrier}. These can be consulted in~\cite{Sageev95, GGTbook14, Arenas2023thesis}, for instance.

A \emph{CAT(0) cube complex} is a simply connected, non-positively curved cube complex. CAT(0) cube complexes are homotopy equivalent to CAT(0) spaces in the metric sense of Cartan--Alexandrov--Toponogov~\cite{BridsonHaefliger, Leary_KanThurston}, and are therefore contractible, which in turn implies that non-positively curved cube complexes are aspherical spaces. While the CAT(0) metric is sometimes useful, in practice one can prove most results about non-positively curved cube complexes without resorting to the CAT(0) metric. Because of this, and since non-positively curved cube complexes correspond to cubical presentations having no relators, and as such, satisfy the $C(9)$ condition in the sense of Definition~\ref{def:Cn}, it is therefore worthwhile to investigate whether the asphericity of non-positively curved cube complexes (or equivalently, the contractibility of their universal covers) can be proved via purely cubical/combinatorial means. This is the purpose of the present section. We stress that alternative proofs of the contractibility of CAT(0) cube complexes, while not present in the literature as far as the author is aware, are well-known to experts.

\begin{theorem}\label{thm:npcaspherical}
Non-positively curved cube complexes are aspherical.
\end{theorem}

\begin{definition}
Let $X$ be a CW complex. Suppose that $c, c'$ are two cubes of $X$ such that $c' \subset c$, 
$c$  is a maximal face of $X$, and no other maximal face of $X$ contains $c'$.
Then $c'$ is  a \emph{free face} of $X$.
A  \emph{collapse} is the removal of all open cells $c''$  such that $ c' \subseteq c \subseteq c'' $ , where $c'$  is a free face.
A cell complex is \emph{collapsible} if there is a sequence of collapses leading to a point.
\end{definition}

All collapsible cell-complexes are contractible~\cite{Whitehead39}, but the converse is not necessarily true, even in the setting of cube complexes: Bing's \emph{house with two rooms}~\cite[pg.~4]{Hatcher2002} is an example of a contractible cube complex that isn't collapsible (and not non-positively curved).

In what follows we will show that finite CAT(0) cube complexes are collapsible. From this, we will be able to deduce that all CAT(0) cube complexes are contractible via a nested subcomplex argument. Note that collapsibility of finite CAT(0) cube complexes follows from more general work of Adiprasito and Benedetti~\cite{AdiprasitoBenedetti20}. However, their proof heavily relies on CAT(0) geometry.

Lurking behind our proof is a form of convexity associated to the $\ell^1$ metric on the $1$-skeleton of a CAT(0) cube complex $X$. We define it below in a manner that elucidates its combinatorial nature.

\begin{definition}[Combinatorial convexity]\label{def:combconv}
A connected subcomplex $Y$ of  a CAT(0) cube complex $X$ is \emph{convex} if, for each $n$-cube $c$  with $n \geq 2$ in  $X$, whenever a corner of $c$ lies in $Y$, then the cube $c$ lies in $Y$.
The \emph{convex hull} $Hull(Z)$ of a subcomplex $Z \subset X$ is the smallest convex subcomplex of $X$ containing $Z$. 
\end{definition}

A consequence of Definition~\ref{def:combconv} and Greendlinger's Lemma~\cite{Greendlinger60} -- a general form of which is stated as Theorem~\ref{thm:tric} below -- is the following observation.

\begin{remark}\label{rmk:convexcat}
A convex subcomplex $Y$ of  a CAT(0) cube complex $X$ is itself a CAT(0) cube complex. 
\end{remark}

The result below can be readily derived from~\cite[2.28]{HaglundGraphProduct}.

\begin{theorem}\label{thm:finitehull}
The convex hull of a finite subcomplex of a CAT(0) cube complex is finite.
\end{theorem}

We also make use of the following proposition, which can be found in~\cite{Sageev95}.

\begin{proposition}\label{prop:hereditarycat} A hyperplane in a CAT(0) cube complex is a CAT(0) cube complex, and separates it into exactly two connected components.
\end{proposition}

\begin{remark} As we are aiming to reprove a basic fact about CAT(0) cube complexes,  some care must be taken to ensure that our arguments are not circular. That is, that none of the result we utilise hinge on the contractibility of CAT(0) cube complexes to begin with. In the case of Theorem~\ref{thm:finitehull} and Proposition~\ref{prop:hereditarycat} above, the proofs only utilise combinatorial convexity, disc diagrams, and the fact that CAT(0) cube complexes are simply connected.
\end{remark}
 
\begin{proof}[Proof of Theorem~\ref{thm:npcaspherical} via collapsibility of finite CAT(0) cube complexes] 
\

\textbf{Finite CAT(0) cube complexes:} We first show, by induction on dimension, that a finite CAT(0) cube complex is either a 0-dimensional cube or has a free face. A connected, 0-dimensional cube complex is necessarily a single vertex; a finite, 1-dimensional CAT(0) cube complex is a finite tree, and therefore has a free face (a leaf). Assume that the result holds for all CAT(0) cube complexes of dimension $< k$.  Let $Y$ be a  $k$-dimensional CAT(0) cube complex, and let $H$ be a hyperplane in $Y$. Then $H$ has dimension less than $k$ and is also a  CAT(0) cube complex, so $H$ has a free face $f_H$. Note that $f_H$ lies in a face $f_Y$ of $Y$, and $f_Y$ cannot be glued to any other cubes of $Y$, as this would imply that $f_H$ extends to a midcube in any such cube. Hence, $f_Y$  is a free face of $Y$.

Now we show, by induction on dimension, that for any CAT(0) cube complex $Y$ and for any free face $f$ of $Y$, there is a sequence of collapses starting with $f$ that results in a subcomplex $Y'$ that is again a CAT(0) cube complex. Note that this claim implies that $Y$ is collapsible, since $Y'$ must have a free face and then the process can be repeated, and terminates in finitely many steps because $Y$ was finite to begin with.  The base case is $dim (Y)=1$, which is immediate as $Y$ is then a tree, so every connected subcomplex of $Y$ is a tree and thus CAT(0).
Now assume that the claim holds for all CAT(0) cube complexes of dimension $< k$ and let $dim (Y)=k$.

 Let $f$ be a free face of $Y$ and $F$ be the maximal face containing $f$, let $H$ be a hyperplane of $Y$ having a free face that is a midcube $m$ of $f$ and note that $m$ is a free face of $H$. Since $dim (H) < dim (Y)$ and $H$ is a CAT(0) cube complex, the induction hypothesis implies that there is a sequence of collapses starting with $m$ that results in a CAT(0) cube complex $H'$, moreover, after repeating this process finitely many times, we may assume that the sequence of collapses terminates at a point $p$ that is a vertex of $H$. We claim that this sequence of collapses extends to a sequence of collapses for the carrier $N(H)$ of $H$ in $Y$ which terminates with an edge $e$ of $Y$.
 
 Indeed, each collapse of $H$  results in a subcomplex that has a free face $m'$, so the corresponding collapse in $Y$ results in a subcomplex that has a free face (the face having $m'$ as a midcube). Thus, the sequence of collapses terminating on $p$ extends to a sequence of collapses terminating in a subcomplex of $Y$ in which $N(H)$ collapses to $N(p)$, which is an edge $e$. Now, since $Y$ is CAT(0), then the interior $N(H)^o$ of $N(H)$ separates $Y$ by Proposition~\ref{prop:hereditarycat}, and each connected component $Z_1,Z_2$  of $Y-N(H)^o$ is CAT(0): each $Z_i$ is non-positively curved since for each vertex $v$ of $Y-N(H)^o$, the link of $v$ in $Y-N(H)^o$ is the restriction $link(v)|_{Y-N(H)^o}$, and is either identical to  $link(v)$, or is obtained from $link(v)$ by removing the flag subcomplex corresponding to $link(v)|_{N(H)^o}$ from it; each $Z_i$ is simply connected because $Y=Z_1 \cup N(H) \cup Z_2$  is simply connected, and is homotopy equivalent to the wedge $Z_1 \vee Z_2$.

Finally, note that $Y'$ is obtained from $Y-N(H)^o$ by reattaching $e$ to it along $(Y-N(H)^o) \cap e$, so it now suffices to check that $link(v)$ is flag for each vertex $v$ of $(Y-N(H)^o) \cap e$ in $Y'$, that is, for the endpoints of $e$. But the links of these vertices are flag in $Y-N(H)^o$, and attaching $e$ only adds a disconnected vertex to each of these links, so the flag condition continues to hold, and $Y'$ is simply connected because both connected components of $Y-N(H)^o$ are, thus $Y'$ is homotopy equivalent to the wedge of $2$ simply connected cube complexes.
 
\textbf{Infinite CAT(0) cube complexes:} 
As $S^n$ is compact, any map $S^n \longrightarrow Y$ for $n \geq 1$ has compact image $Im_Y(S^n)$, so $Im_Y(S^n)$ is contained in a finite subcomplex $Z\subset Y$. Thus,  $Im_Y(S^n) \subset Hull(Z)$, which is finite by Theorem~\ref{thm:finitehull} and Remark~\ref{rmk:convexcat}, and contractible by the argument above. Hence, $S^n \longrightarrow Y$ is nullhomotopic, and $\pi_n(Y)=0$ for all $n\geq 1$.
\end{proof}

\section{Background}\label{sec:back}

\subsection{Cubical small-cancellation  theory}\label{ssec:csm}

Unless otherwise noted, all definitions and results concerning cubical small-cancellation theory recounted in Subsections~\ref{ssec:csm} and~\ref{ssec:diagcub} originate in~\cite{WiseIsraelHierarchy}. 

We begin this section by introducing our main objects of study:

\begin{definition}[Cubical presentation]\label{def:cubpres}
A \textit{cubical presentation} $\langle X|\{Y_i\}  \rangle$ consists of a connected non-positively curved cube complex $X$ together with a collection of local isometries of connected non-positively curved cube complexes $Y_i \overset{\varphi_i} \longrightarrow X$. Local isometries of non-positively curved cube complexes are $\pi_1$-injective, so it makes sense to define the \emph{fundamental group of a cubical presentation} as $\pi_1 X/\nclose{\{\pi_1 Y_i\}}$. By the Seifert-Van Kampen Theorem, this group is isomorphic to the fundamental group of the space $X^*$ obtained by coning off each $Y_i$ in $X$.  
Hereinafter, all auxiliary definitions and results about cubical presentations are in practice statements about their associated coned-off spaces.
\end{definition}

\begin{remark}
A group presentation $\langle a_1,  \ldots, a_s | r_1, \ldots, r_m \rangle$  can be interpreted cubically  by letting $X$ be a bouquet of \textit{s} circles and letting each $Y_i$ map to the path determined by $r_i$. On the other extreme, for every non-positively curved cube complex $X$ there is a ``free" cubical presentation $X^*=\langle X| \ \rangle$ with fundamental group $\pi_1X=\pi_1X^*$.
\end{remark}

\begin{definition}[Elevations] Let $Y \rightarrow X$ be a map where $Y$ is connected, and let $\hat X \rightarrow X$ be a covering map. An \emph{elevation} $\hat Y \rightarrow \hat X$ is a map satisfying
\begin{enumerate}
\item $\hat Y$ is connected,
\item $\hat{Y} \rightarrow Y$ is a covering map, the composition $\hat{Y} \rightarrow Y \rightarrow X$ equals $\hat{Y} \rightarrow \hat X \rightarrow X$, and
\item assuming all maps involved are basepoint preserving, $\pi_1 \hat{Y}$ equals the preimage of $\pi_1 \hat{X}$ in $\pi_1 Y$.
\end{enumerate}
\end{definition}

\begin{conv}\label{conv:elevate} Later in this text we will be dealing exclusively with two kinds of elevations: elevations of a map $Y \rightarrow X$ to the universal cover $\widetilde X \rightarrow X$ and elevations of $Y \rightarrow X$ to a covering space $\hat X \rightarrow X$ that arises as a subspace of $\widetilde{X^*}$.
\begin{enumerate}
\item  Elevations of the first type will be denoted $\widetilde Y \rightarrow X$, since $\widetilde Y$ is indeed a copy of the universal cover of $Y$ in $\widetilde X$. At times, we will distinguish various elevations using the action of $\pi_1X$ on $\widetilde X$. That is,  we choose a base elevation $\widetilde Y$ and  tag a translate $g\widetilde Y$ by the corresponding element $g \in \pi_1X$.
\item Elevations of the second type will be denoted $ Y \rightarrow X$. This choice of notation might seem confusing to the reader now, but it will follow from Theorem~\ref{thm:embeds} that under appropriate small-cancellation conditions --also introduced below, in Definition~\ref{def:Cn}--  elevations $Y \rightarrow X$ to $\hat X \rightarrow X$ are embeddings. As in the case of elevations to $\widetilde X$, we might distinguish various elevations to $\hat X$ using the action of $\pi_1X^*$ on $\widetilde{X^*}$. 
\end{enumerate}
\end{conv}

\begin{definition}[Pieces]\label{def:pieces}
Let $\langle X | \{Y_i\} \rangle$ be a cubical presentation.
An \emph{abstract contiguous cone-piece} $P$ of $Y_j$ in $Y_i$ is an intersection $\widetilde{Y}_j \cap \widetilde{Y}_i$ where $\widetilde{Y}_j, \widetilde{Y}_i$ are fixed elevations to $\widetilde X$ and either $i \neq j$ or where $i = j$
but $\widetilde{Y}_j \neq \widetilde{Y}_i$. Each abstract contiguous cone-piece $P$ induces a map $P \rightarrow Y_i$ which is the composition $P \hookrightarrow \widetilde {Y}_i \rightarrow Y_i$, and a \emph{cone-piece} of $Y_j$ in $Y_i$ is a combinatorial path $p \rightarrow P$ in an abstract contiguous cone-piece of $Y_j$ in $Y_i$.

An \emph{abstract contiguous wall-piece} $P$ of $Y_i$ is an intersection $N(H) \cap \widetilde{Y}_i$ where $\widetilde{Y}_i$ is a fixed elevation and $N(H)$ is
the carrier of a hyperplane $H$ that is disjoint from $\widetilde{Y}_i$. To avoid having to deal with empty pieces, we shall assume that $H$ is dual to an edge with an endpoint on $\widetilde{Y}_i$. As in the case of cone-pieces, each abstract contiguous wall-piece $P$ induces a map $P \rightarrow Y_i$, and a \emph{wall-piece} of $Y_i$ is a combinatorial path $p \rightarrow P$ in an abstract contiguous wall-piece of $Y_i$.

A \emph{piece} is either a cone-piece or a wall-piece.
\end{definition}

\begin{remark} In Definition~\ref{def:pieces}, two elevations of a cone $Y$ are considered identical if they differ by an element of $Stab_{\pi_1X}(\widetilde Y)$. This is in keeping with the conventions of classical small cancellation theory, where overlaps between a relator and any of its cyclic permutations are not regarded as pieces.  
\end{remark}

\begin{definition}
Let $Y \rightarrow X$ be a local isometry. $Aut_X(Y)$ is the group of combinatorial automorphisms $\psi: Y \rightarrow Y$ such that the diagram below is commutative:
\[\begin{tikzcd}
Y \arrow[r, "\psi"] \arrow[rd] & Y \arrow[d] \\
                               & X          
\end{tikzcd}\]
If $Y$ is simply connected, then $Aut_X(Y)$ is equal to $Stab_{\pi_1X}(Y)$. In general, $Aut_X(Y)\cong (N_{Aut_X(\widetilde Y)}\pi_1Y)/\pi_1Y$, where $N_G(H)$ is the normaliser of $H$ in $G$. 

See~\cite{ArHag2021} for a detailed discussion.
\end{definition}

\begin{conv} Let $X^*=\langle X | \{Y_i\}_I \rangle$. For convenience we assume, as is standard to assume in this framework~\cite[3.3]{WiseIsraelHierarchy}, that $\pi_1Y_i$ is normal in $Stab_{\pi_1X}(\widetilde Y_i)$ for each $i \in I$. Throughout, we will assume also that $X$ is finite dimensional and locally finite, and that each $Y_i$ is connected.
\end{conv}

For the purposes of this paper, we will need to impose further restrictions on a cubical presentation; these are described below.

\begin{definition}\label{def:symmetric} 
 A local-isometry $Y  \rightarrow X$ with $Y$ connected and superconvex is symmetric if for each component $K$ of the fibre product $Y\otimes_{_X} Y$, either $K$ maps isomorphically to
each copy of $Y$, or $[\pi_1Y : \pi_1K]=\infty$. A cubical presentation $\langle X | \{Y_i\} \rangle$ is \emph{symmetric} if each $Y_i \rightarrow X$ is symmetric.
\end{definition}

Superconvexity is defined in~\cite[2.35]{WiseIsraelHierarchy}, and also in~\cite[2.4]{Arenas2023}. We refrain from stating the definition here, since we will only use it implicitly when applying Lemma~\ref{lem:symm} in Section~\ref{sec:main}.

\begin{definition}(Commensurator)
The \emph{commensurator} $C_G(H)$ of a subgroup $H$ of $G$ is the subgroup
 $C_G(H)=\{g \in G: [H :H^g \cap H] < \infty \text{ and } [H^g :H^g \cap H] < \infty  \}$.
\end{definition}

In~\cite[8.12]{WiseIsraelHierarchy}, the following is observed:

\begin{lemma}\label{lem:symm} Let $X^*=\langle X | \{Y_i\} \rangle$ be a cubical presentation. Then $X^*$ is symmetric if and only if $C_{\pi_1X}(\pi_1 Y_i)=Stab_{\pi_1X}(\widetilde Y_i)$, $[Stab_{\pi_1X}(\widetilde Y_i):\pi_1Y_i]< \infty$, and $\pi_1Y_i \vartriangleleft Stab_{\pi_1X}(\widetilde Y_i)$ for each  $Y_i \rightarrow X$.
\end{lemma}

To the best of our knowledge, the following notion was introduced in~\cite{ArHag2021}, where it is used in the course of proving acylindrical hyperbolicity for certain cubical presentations satisfying strong cubical small-cancellation conditions.

\begin{definition}\label{def:minimal}  A cubical presentation $\langle X | \{Y_i\}_I \rangle$ is \emph{minimal} if the following holds for each $i \in I$: let $\widetilde Y_i \rightarrow Y_i$ be the universal cover, and let $\widetilde{Y_i} \rightarrow Y_i \rightarrow X$ be an elevation of $Y_i \rightarrow X$. Fix a basepoint $x_0$ in $X$. Then $Stab_{\pi_1(X, x_0)}(\widetilde{Y_i})=\pi_1 (Y_i, y_{i_0})$.
\end{definition}

Minimality generalises prohibiting relators that are proper powers in the classical small-cancellation case. In our setting, minimality is used to avoid the ``obvious'' torsion that could be created in a quotient $\pi_1X/\langle \langle \{\pi_1 Y'_i\} \rangle \rangle$ if the  $Y'_i \rightarrow X$  are themselves non-trivial finite covers $Y'_i \rightarrow Y_i \rightarrow X$. For instance, if $Y$ is a finite degree covering of $X$, the cubical presentation $X^*=\langle X | Y \rangle$ will have finite fundamental group, even if it satisfies the $C(9)$ condition defined below (or any other small cancellation condition).

%
%

\begin{definition}\label{def:Cn}
A cubical presentation $X^*$ satisfies the \textit{$C(n)$ small-cancellation condition} if no essential closed path $\sigma \rightarrow Y_i$ is the concatenation of fewer than $n$ pieces.
\end{definition}

An analogue of the $C'(\frac{1}{n})$ condition can also be defined in this setting, but we won't need it in the present work. We remark only that, as in the classical case, the cubical $C'(\frac{1}{n})$ condition implies the cubical $C(n+1)$ condition, but the implication is very much not reversible: regardless of the choices of $n$ and $n'$, the cubical $C(n)$ condition does not necessarily imply the cubical $C'(\frac{1}{n'})$ condition.

\subsection{Diagrams in cube complexes and cubical presentations}\label{ssec:diagcub}

In this section we introduce disc and spherical diagrams, and the general related jargon -- the analysis of the possible diagrammatic behaviours in  $C(9)$ cubical presentations, and the use of the tools already available in this context, will be the main ingredients utilised in Section~\ref{sec:main}.

\begin{definition}
A map $f: X \longrightarrow Y$ between 2-complexes is \emph{combinatorial} if it maps cells to cells of the same dimension.
A complex is \emph{combinatorial} if all attaching maps are combinatorial, possibly after subdividing the cells. 
\end{definition}

\begin{definition}[Disc diagram] A \emph{disc diagram} $D$ is a compact contractible combinatorial 2-complex, together with an embedding $D \hookrightarrow S^2$. The \emph{boundary path} $\partial D$ is the attaching map of the 2-cell at infinity. A \emph{disc diagram in a complex X} is a combinatorial map $D \rightarrow X$.
A \emph{square disc diagram} is a disc diagram that is also a cube complex (though \textbf{not necessarily} non-positively curved!). Note that any disc diagram in a cube complex is a square disc diagram. 
\end{definition}

We introduce some of the phenomena that may arise in square-disc diagrams;  unlike in the general setting of $2$-complexes, problematic behaviour of diagrams can be described and classified precisely in the cubical setting.

\begin{definition}[Square-disc behaviours]\label{def:squarebee}
 A \emph{dual curve} in a square disc diagram is a path that is a concatenation of midcubes. The 1-cells crossed by a dual curve are \emph{dual} to it.
 A \emph{bigon} is a pair of dual curves that cross at their first and last midcubes. A \emph{monogon} is a single dual curve that crosses itself at its first and last midcubes. A \emph{nonogon} is a single dual curve of length $\geq1$ that starts and ends on the same dual 1-cell, thus it corresponds to an immersed cycle of midcubes. A \emph{spur} is a vertex of degree $1$ on $\partial D$.

A \emph{corner} in a diagram $D$ is a vertex $v$ that is an endpoint of consecutive edges $a,b$ on $\partial D$ lying in a square $s$.
 A \emph{cornsquare} -- short for ``generalised corner of a square''-- consists of a square $c$ and dual curves $p, q$ emanating from consecutive edges $a, b$ of $c$ that terminate on consecutive edges $a',b'$ of $\partial D$. The \emph{outerpath} of the cornsquare is the path $a'b'$ on $\partial D$. 
A \emph{cancellable pair} in $D$ is a pair of $2$-cells $R_1, R_2$ meeting along a path $e$ such that the following diagram commutes:
 
 \[\begin{tikzcd}
                         & e \arrow[ld] \arrow[rd] &                          \\
\partial R_1 \arrow[rd] \arrow[rr]  &                         & \partial R_2 \arrow[ld] \\
                         & X                       &                         
\end{tikzcd}\]

 A cancellable pair leads to a smaller area disc diagram via the following procedure: cut out $e\cup Int(R_1)\cup Int(R_2)$ and then glue together the paths $\partial R_1-e$ and $\partial R_2-e$ to obtain a diagram $D'$ with $\area(D')=\area(D)-2$ and $\partial D'=\partial D$.
\end{definition}

By performing the procedure just described to cancellable pairs, diagrams in non-positively curved cube complexes can often be simplified to avoid certain pathologies.

\begin{lemma}\label{lem:discpatho}\cite[2.3+2.4]{WiseIsraelHierarchy}	Let $D \rightarrow X$ be a disc diagram in a non-positively curved cube complex. If $D$ contains a bigon or a nonogon, then there is a new diagram $D'$ having the same boundary path as $D$, so $\partial D' \rightarrow X$ equals $\partial D \rightarrow X$, and such that $Area(D')\leq Area(D)-2$.
Moreover,  no disc diagram in $X$ contains a monogon, and if $D$ has minimal area among all diagrams with boundary path $\partial D$, then $D$ cannot contain a bigon nor a nonogon.
\end{lemma}

Some additional terminology is necessary when describing diagrams in cubical presentations:

\begin{definition}
Recall that the coned-off space $X^*$ introduced in Definition~\ref{def:cubpres} consists of $X$ with a cone on $Y_i$ attached to $X$ for each $i$. The vertices of the cones on $Y_i$'s are the \emph{cone-vertices} of $X^*$. The cellular structure of $X^*$ consists of all the original cubes of $X$, and the pyramids over cubes in $Y_i$ with a cone-vertex for the apex. Let $D \rightarrow X^*$ be a disc diagram in a cubical presentation. The vertices in $D$ which are mapped to the cone-vertices of $X^*$ are the \emph{cone-vertices} of $D$. 
Triangles in $D$ are naturally grouped into cyclic families meeting around a cone-vertex.
Each such family forms a subspace of $D$ that is a cone on its bounding cycle. A  \emph{cone-cell} of $D$ is a cone that arises in this way.
To simplify the theory, when analysing diagrams in a cubical presentation we ``forget'' the subdivided cell-structure of a cone-cell $C$ and regard it simply as a single $2$-cell. 
\end{definition}

A situation that may occur with diagrams in coned-off spaces is that two cone-cells might come from the same coned-off relation $Y$. When this happens, it is often possible to fuse these two cone-cells together into a single cone-cell, as explained below. This simplification does not arise in the ``purely cubical" setting, but will be useful for analysing diagrams in cubical presentations.

\begin{definition}A pair of cone cells $C,C'$ in $D$ is \emph{combinable} if they  map to the same cone $Y$ of $X^*$ and $\partial C$ and $\partial C'$ both pass through a vertex $v$ of $D$, and  map to closed paths
at the same point of $Y$ when regarding $v$ as their basepoint.

As the name suggests, such a pair can be combined to simplify the diagram by replacing the pair with a single cone-cell mapping to $Y$ and whose boundary is the concatenation $\partial C \partial C'$.
\end{definition}

\begin{definition}\label{def:reduced} A \emph{disc diagram} $D\rightarrow X^*$ is \emph{reduced} if the following conditions hold:
\begin{enumerate}
\item \label{it:r1} There is no bigon in a square subdiagram of $D$.
\item \label{it:r2} There is no cornsquare whose outerpath lies on a cone-cell of $D$. 
\item \label{it:r3} There does not exist a cancellable pair of squares. 
\item \label{it:r4} There is no square $s$ in $D$ with an edge on a cone-cell $C$ mapping to the
cone $Y$, such that $(C\cup s)\rightarrow X$ factors as $(C\cup s) \rightarrow Y \rightarrow X$.
\item \label{it:r5} For each internal cone-cell $C$ of $D$ mapping to a cone $Y$, the path $\partial C$ is essential in $Y$.
\item \label{it:r6} There does not exist a pair of combinable cone-cells in $D$.
\end{enumerate}
\end{definition}

A closely-related notion to that of disc diagram reducibility is that of complexity:

\begin{definition}\label{def:mincomp}
The \emph{complexity} $Comp(D)$ of a disc diagram $D\rightarrow X^*$ is the ordered pair $(\#\text{Cone-cells}, \#\text{Squares})$.
We order the pairs lexicographically: namely $(\#C,\#S) < (\#C',\#S')$ whenever $\#C<\#C'$ or $\#C=\#C'$ and $\#S < \#S'$.
A disc diagram $D\rightarrow X^*$ has \emph{minimal complexity} if no disc diagram $D'\rightarrow X^*$ having $\partial D =\partial D'$ has $Comp(D')< Comp(D)$.
\end{definition}

Since the objective of this work is to understand the second homotopy group of a cubical presentation, it will come as no surprise that in addition to disc diagrams, their spherical analogues will also play a key role in this endeavour.

\begin{definition}[Spherical diagram] A \emph{spherical diagram} $\Sigma$ is a compact simply-connected combinatorial 2-complex, together with a homeomorphism $\Sigma \hookrightarrow S^2$. A \emph{spherical diagram in a complex X} is a combinatorial map $\Sigma \rightarrow X$.
We define \emph{reduced} spherical diagrams and \emph{complexity} in the same way as in Definitions~\ref{def:reduced} and~\ref{def:mincomp}, replacing all instances of ``disc diagram'' by ``spherical diagram''. A spherical diagram has \emph{minimal complexity} if no spherical diagram $\Sigma'\rightarrow X^*$ homotopic to $\Sigma$ has $Comp(\Sigma')< Comp(\Sigma)$.
\end{definition}

All diagrams in this paper are either square disc or square spherical diagrams, or are disc or spherical diagrams in the coned-off space associated to a cubical presentation.

\begin{remark}\label{rmk:mincompisreduced}
If a disc diagram or spherical diagram has minimal complexity, then it is also reduced, as any of the pathologies in Definition~\ref{def:reduced} would indicate a possible complexity reduction. For instance, a cornsquare whose outerpath lies on a cone-cell can be absorbed, after a sequence of hexagon moves -- each of which corresponds to pushing a hexagon on one side of a $3$-cube to obtain the hexagon on the other side -- into said cone-cell; a cancellable pair of squares can be cancelled as described in Definition~\ref{def:squarebee}; a square bigon leads to a cancellable pair of squares after a sequence of hexagon moves; and a pair of combinable cone-cells in $D$ can be combined to reduce the number of cone-cells by $1$.  A detailed case-by-case analysis of all possible reductions can be found in~\cite[3.e]{WiseIsraelHierarchy}. 
\end{remark}

\begin{definition}[Shell]
A \emph{shell} of $D$ is a 2-cell $C \rightarrow D$
whose boundary path $\partial C \rightarrow D$ is a concatenation $QP_1 \cdots P_k$ for some $k \leq 4$
where $Q$ is a boundary arc in $D$ and $P_1, \ldots , P_k$ are non-trivial pieces in the interior of $D$.
The arc $Q$ is the \emph{outerpath} of $C$ and the concatenation $S:=P_1 \cdots P_k$ is the \emph{innerpath} of $C$.  
\end{definition}

\begin{remark}
Note that if a cubical presentation $X^*$ satisfies the $C(n)$ condition and $D$ is a minimal complexity diagram, then the outerpath of a shell in $D$ is the concatenation of $\geq n-4$ pieces.
\end{remark}

\begin{figure}
\centerline{\includegraphics[scale=0.23]{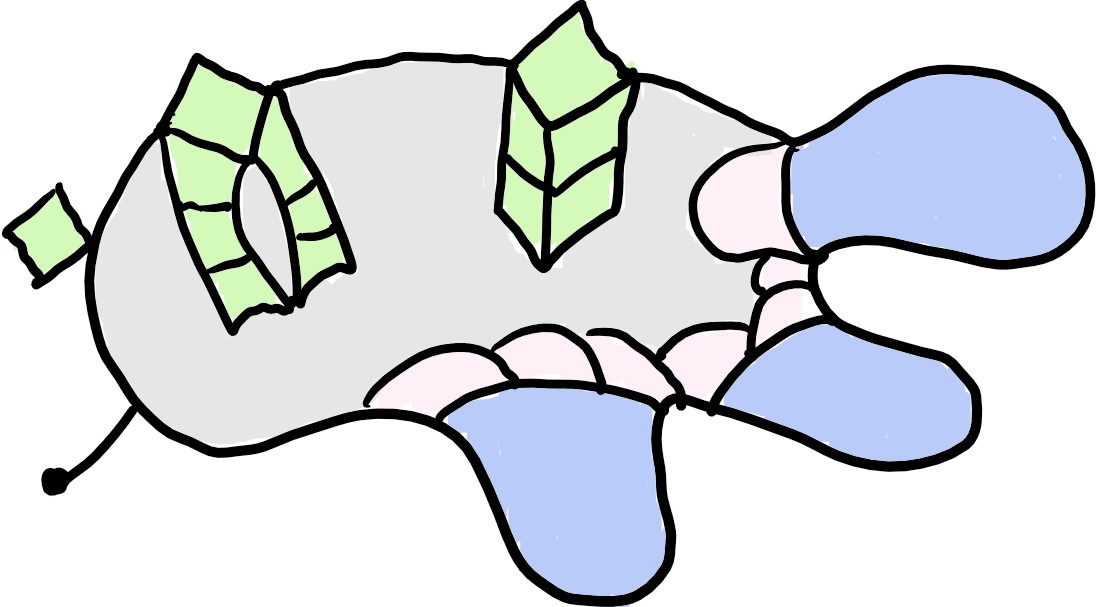}}
\caption{A potential disc diagram with a range of features: a spur, shells, corners, cornsquares and a cut-vertex.}
\label{fig:exdiagram}
\end{figure} 

In almost all forms of small-cancellation theory, the main technical result that facilitates proving theorems is a form of diagram classification. In the classical setting, this is known as Greendlinger's Lemma~\cite{Greendlinger60}, in the cubical setting this is  known as diagram trichotomy, or as the ``Fundamental Lemma". We state it in a simplified form that is sufficient for our applications, and which we shall call \emph{diagram dichotomy}:

\begin{theorem}[Diagram Dichotomy]\label{thm:tric}\cite{JankiewiczSmallCancellation}
Let $X^*= \langle X| \{Y_i\}\rangle$ be a cubical presentation satisfying the $C(9)$ condition, and let $D\to X^*$ be a disc diagram. Then one of the following holds:
\begin{itemize}
\item $D$ consists of a single cell,
\item $D$ has at least two shells and/or corners and/or spurs.
\end{itemize}
\end{theorem}

\begin{theorem}\cite{JankiewiczSmallCancellation}\label{thm:embeds}
Let $X^*=\langle X|\{Y_i\}\rangle$ be a cubical presentation satisfying the $C(9)$ condition. Then each $Y_i$ embeds in the universal cover
$\widetilde{X^*}$ of the coned-off space $X^*$.
\end{theorem}

\subsection{Asphericity}

We will require the following fundamental fact about combinatorial paths in $2$-complexes.

\begin{theorem}[The Van-Kampen Lemma]\label{thm:VK}
Let $X$ be a combinatorial $2$-complex. Let $P \rightarrow X^1$ be a closed combinatorial path. Then $P$ is nullhomotopic if and only if there exists a disc diagram $D$ in $X$ with $\partial D \cong P$ so that there is a commutative diagram:

\[\begin{tikzcd}
\partial D \arrow[r] \arrow[d] & D \arrow[d] \\
P \arrow[r]                      & X          
\end{tikzcd}\]

\end{theorem}

For our purposes, the spherical version of Van-Kampen's Lemma will also be extremely useful. That is, we need to be able to assert that all elements of $\pi_2$ have diagram representatives:

\begin{theorem}[The Spherical Van-Kampen Lemma]\label{thm:SVK}
Let $X$ be a combinatorial $2$-complex. Then every homotopy class of maps  $S^2 \rightarrow X$ is represented by a spherical diagram  $\Sigma \rightarrow X$. 
\end{theorem}

Theorem~\ref{thm:SVK} is presented (with or without a proof) in slightly different ways in various sources. The version we use can be extracted from~\cite[Section 2]{Fenn83} via the viewpoint of ``pictures'', which are planar graph representations of homotopy elements and are dual to spherical diagrams. 

\section{Main theorem}\label{sec:main}

We commence this section with a definition.

\begin{definition}[Classifying space for proper actions]
Let $G$ be a group. A \emph{classifying space for proper actions} for $G$ is a $G$-CW-complex $\underbar EG$ satisfying: 
\begin{enumerate}
\item every cell stabiliser is finite,
\item for each finite subgroup $H<G$, the fixed point space  $\underbar EG^H$ is contractible.
\end{enumerate}
\end{definition}

Every group $G$ admits an $\underbar EG$, and  all models for $\underbar EG$ are $G$-homotopy equivalent; we also have the inequality $cd_\rationals(G)\leq dim (\underbar EG)$, for any group $G$ admitting a finite dimensional $\underbar EG$. 
More details, and a general construction of classifying spaces for proper actions, can be found in~\cite{Luck05}.

To prove our asphericity result, we must slightly simplify the coned-off space $X^*$ associated to a cubical presentation. Intuitively, we do this in  $\widetilde{X^*}$ by ``collapsing" or ``squashing together" various cones into a single one whenever their base spaces correspond to  elevations of a $Y_i \rightarrow X$ having the same preimage in $\widetilde X$; at this point, we remind the reader of Convention~\ref{conv:elevate}. 

The precise construction is as follows:

\begin{construction}\label{def:reduction}
Let $X^*= \langle X |\{Y_i\}_{i \in I} \rangle$ be a cubical presentation and consider the universal cover $\widetilde{X^*}$ of the coned-off space. 
Note that $X$ is a subspace of $X^*$, so the preimage $\hat X$ of $X$ in $\widetilde{X^*}$ is a covering space of $X$, namely the regular cover corresponding to $ker(\pi_1X \rightarrow \pi_1X^*)$. Consider the universal cover $\widetilde X $ of $X$.
For each $i \in I$, and for any fixed elevation $\widetilde Y_i \rightarrow \widetilde X$ of $Y_i$,  we have that $\pi_1 (Y_i,y_{i_0}) < Stab_{\pi_1(X,x_0)}(\widetilde Y_i)$. If $X^*= \langle X |\{Y_i\}_{i \in I} \rangle$ is minimal, then by definition $\pi_1 (Y_i,y_{i_0}) = Stab_{\pi_1(X,x_0)}(\widetilde Y_i)$, but in general this is not the case.

Let $\{g_\ell\pi_1 Y_i\}$ be coset representatives of  $\pi_1 (Y_i,y_{i_0})$ in $Stab_{\pi_1(X,x_0)}(\widetilde Y_i)$. 
The elevations $\{g_\ell Y_i \rightarrow \hat X\}$ have the same image in $\hat{X}\subset \widetilde{X^*}$, so their cones are all isomorphic in $ \widetilde{X^*}$. 
Thus, there is a quotient $\cup_{\ell}g_\ell C(Im_{\hat X} (Y_i)) \rightarrow C(Im_{\hat X} (Y_i))$ where all cones over $Im_{\hat X}(Y_i)$ are identified to a single cone. This extends to a quotient $\widetilde{X^*} \rightarrow \bar X^*$, which we call the \emph{reduced space} of $\widetilde{X^*}$.
\end{construction}

\begin{remark}[Cubically presenting the trivial group] 
Another way to think about $\bar X^*$ is as the cubical presentation $ \langle \hat X |\{Im_{\hat X} (Y_i) \} \rangle$, where the $Y_i$ range over all $i \in I$ and over all elevations of $Y_i$ with distict image. Of course, $\hat X$ is not compact, and there are infinitely many $Y_i$'s, so this is an ``infinitely generated'' cubical presentation with infinitely many ``relators''. It presents, in fact, the trivial group, as we show in Lemma~\ref{lem:scbar}. It is immediate from Construction~\ref{def:reduction} that if $\widetilde{X^*}$ satisfies the $C(n)$ condition for $n >0$, then so does $\bar X^*$.

Thus, there are no hidden technicalities in applying the theory from Section~\ref{sec:back} to  $\bar X^*$ rather than to $\widetilde{X^*}$. In particular, it makes sense to talk about minimal complexity diagrams and pieces, and Diagram Dichotomy and Theorem~\ref{thm:embeds} all hold for $\bar X^*$.
\end{remark}

As mentioned above, the following is a quick but important observation:

\begin{lemma}\label{lem:scbar}
 $\bar X^*$ is simply-connected.
\end{lemma}

\begin{proof}
Choosing a representative $g_0 C(Im_{\hat X} (Y_i))$ for each collection $\{g_\ell C(Im_{\hat X} (Y_i))\}_{g_\ell}$ described above, we may view the reduced space $\bar X^*$ as a subspace of $\widetilde{X^*}$. Thus, $\bar X^*$ is a retract of $\widetilde{X^*}$, and $\pi_1\bar X^*$ injects into $\pi_1\widetilde{X^*}$, and is therefore trivial.
\end{proof}

We can now state and prove our main technical theorem.

\begin{theorem}\label{thm:cub2} Let $X^*= \langle X |\{Y_i\} \rangle$ be a cubical presentation that satisfies the $C(9)$ condition. Let $\bar X^*$ be the reduced space of $\widetilde X^*$. 
Then $\pi_2\bar X^*=0$.
\end{theorem}

\begin{proof}
Let $\Sigma \hookrightarrow \bar X^*$ be a minimal complexity spherical diagram, where the minimum is taken over all  spherical diagrams representing a fixed non-zero homotopy class in $\pi_2\bar X^*$. If $\Sigma$ has no cone-cells, then it lies on the cubical part of  $\bar X^*$ and it is therefore contractible by Theorem~\ref{thm:npcaspherical}. Thus, we may assume that $\Sigma$ has a cone-cell.

Let $C_\infty$ be a cone-cell in $\Sigma$ and let $D_0$ be obtained from $\Sigma$ by removing from it the interior of $C_\infty$. Note that $D_0$ is a disc diagram. Indeed, since $\Sigma$ has minimal complexity, then $C_\infty$ embeds in $\Sigma$, as by Theorem~\ref{thm:embeds} the cone of $\bar X^*$ containing  $C_\infty$ embeds in $\bar X^*$, so a non-embedded cone-cell in $\Sigma$ would have to be nullhomotopic, contradicting Condition~\ref{it:r5} of Definition~\ref{def:reduced}. Moreover, $D_0$   has minimal complexity as any complexity reduction in $D_0$ would lead to a complexity reduction in $\Sigma$.

By Theorem~\ref{thm:tric}, either 
\begin{enumerate}
\item $D_0$ is a single vertex or a single cone-cell,
\item $D_0$ has at least 2 shells and/or corners and/or spurs.
\end{enumerate}
The first case would imply that $\Sigma$ consists exactly of two cone-cells $C_1$ and $C_2$ with $\partial C_1=\partial C_2=\sigma$. Assume $\partial C_1 \rightarrow Y_1$ and $\partial C_2 \rightarrow Y_2$. If $Y_1=Y_2$, then  $\Sigma \rightarrow \bar X^*$ factors as $\Sigma \rightarrow C(Y_1) \rightarrow \bar X^*$. But $C(Y_1)$ is contractible, because it is a cone, so $\Sigma$ is nullhomotopic in $C(Y_1)$, and therefore also in $\bar X^*$, which contradicts the choice of $\Sigma$.

Otherwise, $Y_1\neq Y_2$. If  $\sigma$ is essential in either $Y_1$ or $Y_2$, then either  $\sigma$ is a piece, and thus an essential path that is the concatenation of $<9$ pieces, contradicting the $C(9)$ condition,  or $\sigma$ is not a piece, which implies that $Y_1$ and $Y_2$ are elevations of the same $Y_i$ that differ by an element of $Stab_{\pi_1X}(\widetilde Y)$. In this case, $Y_1$ and $Y_2$ have the same image, so $C(Y_1)$ and $C(Y_2)$ are identified in $\bar X^*$, and thus $C_1=C_2$, again contradicting the minimal complexity of $\Sigma$.  If  $\sigma$ is not essential in neither $Y_1$ nor $Y_2$, then there are disc diagrams $D_1 \rightarrow Y_1, D_2 \rightarrow Y_2$ with $\partial D_1=\partial C_1$ and $\partial D_2=\partial C_2$, which together bound a spherical diagram $\Sigma' \rightarrow \hat X$. As $C(Y_1)$ and $C(Y_2)$ are contractible, $\Sigma'$ and $\Sigma$ are homotopic. Since  $\hat X$ is aspherical -- it is a covering space of a non-positively curved cube complex, and thus also non-positively curved -- then  $\Sigma'$ is nullhomotopic in $\hat X$, and therefore also in $\bar X^*$. This in turn implies that $\Sigma$ is nullhomotopic in $\bar X^*$.

In the third case, the presence of spurs immediately implies that $D_0$ is not a minimal complexity diagram. Likewise, the presence of a corner on $\partial D_0$ implies the presence of a corner on $\partial C_\infty$, contradicting that the spherical diagram $\Sigma$  has minimal complexity, as such a corner could be absorbed into $C_\infty$, reducing the number of squares in $\Sigma$ by Condition~\ref{it:r2} of Definition~\ref{def:reduced}. In particular, this implies that $D_0$ cannot be a square disc diagram.
We can therefore conclude that $D_0$ is a disc diagram with at least two shells -- in fact, all we need to use now is that $D_0$ has at least one shell.

Let $C$ be a shell of $D_0$. Then the inner-path of $C$ is at most $4$ pieces, and the outerpath of $\partial C$ coincides with a subpath of $\partial C_\infty$ in $\Sigma$. Thus, the outerpath of $C$ is a single piece, contradicting the $C(9)$ condition.
\end{proof}

We can now deduce asphericity for the reduced space $\bar X^*$ associated to a low-dimensional $C(9)$ cubical presentation, and in particular for $X^*$ when the cubical presentation is minimal:

\begin{corollary}\label{cor:lowdim}
Let $X^*=\langle X | \{Y_i\} \rangle$ be a cubical presentation that satisfies the $C(9)$ condition.  If $dim(X) \leq 2$  and $\max \{cd(\pi_1Y_i)\}=1$, then $\bar X^*$ is contractible. 
\end{corollary}

\begin{proof}
As explained in Construction~\ref{const:covers},  the preimage $\hat X$ of $X$ in  $\widetilde{X^*}$ is a covering space of $X$,  corresponding to $ker(\pi_1X \rightarrow \pi_1X^*)$. 
By Theorem~\ref{thm:embeds}, each $Y_i$ embeds in $\widetilde{X^*}$, and thus also in the quotient  $\bar X^*$.
Viewing  $\bar X^*$ as a cubical presentation, $\bar X^*$ decomposes as the union of its cubical part $\hat X$ and the cones over all elevations of $Y_i$'s with distinct images in $\widetilde{X^*}$. 
To simplify notation, let $$\bigsqcup_i \bigsqcup_{gStab_{\pi_1(X,x_0)}(\widetilde Y_i) \in \pi_1X/Stab_{\pi_1(X,x_0)}(\widetilde Y_i)}  gY_i:= \mathbf{Y} $$and $$ \bigsqcup_i \bigsqcup_{gStab_{\pi_1(X,x_0)}(\widetilde Y_i) \in \pi_1X/Stab_{\pi_1(X,x_0)}(\widetilde Y_i)}  gC(Y_i):= \mathbf{C(Y)},$$ 

And note that $$\hat X \cap \mathbf{C(Y)}= \mathbf{Y}.$$

We have the Mayer-Vietoris sequence:
\[
    \begin{tikzcd}[arrows=to]
        \cdots \rar & H_n(\mathbf{Y}) \rar & H_n(\hat X) \oplus H_n (\mathbf{C(Y)}) \rar & H_n(\bar X^*) \rar & \hphantom{0}\\
        \hphantom{\cdots} \rar 
        & H_{n-1}(\mathbf{Y}) \rar 
        & \makebox[\widthof{$H_n(A) \oplus H_n(B)$}][c]{$\cdots\hfill \cdots$} \rar
        &  H_0(\hat X^*) \rar & 0
    \end{tikzcd}
\]

Since $\mathbf{Y}$ is homotopy equivalent to a graph, then $H_n(\mathbf{Y})=0$ whenever $n \geq 2$, so we get isomorphisms $H_n(\hat X)  \cong  H_n(\bar X^*)$ for each $n \geq 2$, since $C(gY_i))$ is contractible for each $g \in \pi_1X^* $ and $i \in I$. Now, $H_2(\bar X^*)\cong \pi_2\bar X^*=0$ by Theorem~\ref{thm:cub2} and Hurewicz's Theorem, and $H_3(\bar X^*)=0$ since the sequence is exact and the terms on the right and left of $H_3(\bar X^*)$ are equal to zero. Since $\bar X^*$ is $2$-connected, we may apply Hurewicz again to conclude that $\pi_3\bar X^*=0$. 

The proof is now finished:  since by Theorem~\ref{thm:cub2}, $\pi_1 \bar X^*=0$, and  $\bar X^*$ has no cells of dimension $\geq 3$, then $H_n(\bar X^*)\cong \pi_n\bar X^*=0$ for all $n \in \naturals$.
\end{proof}

The first of our two theorems is now established:

\begin{theorem}\label{thm:A}
Let $X^*=\langle X | \{Y_i\}_{i \in I} \rangle$ be a minimal cubical presentation that satisfies the $C(9)$ condition. Let $\pi_1 X^*=G$. If $dim(X)\leq 2$ and each $Y_i$ is homotopy equivalent to a graph, then $X^*$ is a $K(G,1)$, so $G$ is torsion-free and $gd(G)\leq 2$.
\end{theorem}

\begin{proof}
This is a direct consequence of Corollary~\ref{cor:lowdim}, since then, by minimality, $\bar X^*=\widetilde{X^*}$, so $\widetilde{X^*}$ is contractible. 
\end{proof}

We point out that neither Theorem~\ref{thm:cub2} nor Corollary~\ref{cor:lowdim} require the cubical presentation $X^*=\langle X | \{Y_i\}_{i \in I} \rangle$ to have finitely many relators, or for these relators to be compact. Thus, Theorem~\ref{thm:A} holds even if the $Y_i$ have infinitely-generated fundamental group, and even if $|I|= \infty$. However, this is not the case for the results that follow: for the remainder of this section, we adopt the following convention.

\begin{conv} Let $X^*=\langle X | \{Y_i\}_I \rangle$ be a cubical presentation. If $X^*$ is \emph{not} minimal, then we always assume that $|I|< \infty$, and that each $Y_i$ is compact.
\end{conv}

\begin{lemma}\label{lem:stabs} Let $X^*=\langle X | \{Y_i\}_{i \in I} \rangle$ be a symmetric cubical presentation. 
Then  $\bar X^*$ is a $\pi_1 X^*$-CW-complex, and if $v$ is a cone-vertex of $\bar X^*$ corresponding to a cone over some $ Y_i$, then $Stab_{\pi_1X^*}(v)=Stab_{\pi_1X}(\widetilde Y_i)/\pi_1{Y_i}$.
\end{lemma}

\begin{proof}
The group $\pi_1 X^*$ acts on $\widetilde {X^*}$ by permuting the cones. The action is a covering space action, and in particular is free, and induces an action on $\bar X^*$. 
For each cone-vertex $v$ in $\bar X^*$, $Stab_{\pi_1X^*}(v)=Stab_{\pi_1X^*}(C(Im_{\hat X} (\widetilde Y_i))=\{g_\ell\pi_1 Y_i\}$ where $C(Im_{\hat X} (\widetilde Y_i))$ is the cone with $v$ as its cone-vertex, and $\{g_\ell\pi_1 Y_i\}$ are left coset representatives of  $\pi_1 (Y_i,y_{i_0})$ in $Stab_{\pi_1(X,x_0)}(\widetilde Y_i)$.  This is exactly the quotient $Stab_{\pi_1X}(\widetilde Y_i)/\pi_1{Y_i}$.
\end{proof}

\begin{lemma}\label{lem:finitestab}
If $X^*=\langle X | \{Y_i\} \rangle$ is a symmetric cubical presentation, 
then the action of $\pi_1X^*$ on $\bar X^*$ has finite cell-stabilisers. 
\end{lemma}

\begin{proof}
All stabilisers of cells in $\hat X \subset \bar X^*$ are trivial; since the cubical presentation is symmetric, then $[Stab_{\pi_1X}(\widetilde Y):\pi_1Y]< \infty$, so Lemma~\ref{lem:stabs} implies that the stabilisers of cone-vertices are finite.
\end{proof}

So far we have shown that if $X^*=\langle X | \{Y_i\} \rangle$ satisfies the $C(9)$ condition, then  $\bar X^*$ is contractible, and that if, in addition, $X^*$ is symmetric, then the action of $\pi_1X^*$ on $\bar X^*$  has finite cell-stabilisers. To show that $\bar X^*$ is a classifying space for proper actions for $\pi_1X^*$, we still must prove that every finite subgroup $H <\pi_1X^*$ fixes a point in $\bar X^*$, and that this point is unique. 
Note that since $\bar X^*$ is aspherical and finite-dimensional, then the action of a finite subgroup $H$ on $\bar X^*$ cannot be free, as otherwise the quotient $\bar X^*/H$ would be a classifying space for $H$. Since we will invoke it later on, we record this observation as a lemma:

\begin{lemma}\label{lem:nonfree}
Let $Z$ be an aspherical, finite-dimensional cell-complex and let $G$ be a group acting on $Z$ by combinatorial isometries. If $H < G$ is finite, then  the action of $H$ on $Z$ is not free.
\end{lemma}

We emphasize that this does not necessarily imply that such an action has a global fixed point.

\begin{lemma}\label{lem:iffix}
If  a non-trivial subgroup $H <\pi_1X^*$ fixes a point in $\bar X^*$, then $H$ is conjugate into $Stab_{\pi_1X^*}(ImY_i)$ for some elevation $Y_i\rightarrow \hat X$ of some $Y_i \rightarrow X$.
\end{lemma}

\begin{proof}
Let  $H$ be a non-trivial subgroup of $\pi_1X^*$. Since the action of $H$ on $\widetilde{X^*}$ is free and $\widetilde{X^*}$ coincides with $\bar X^*$ outside of the cones, then a fixed point $\xi$ under the action of $H$ on $\bar X^*$ must be a cone-vertex, which in turn, by Construction~\ref{const:covers}, corresponds to some left coset  $gStab_{\pi_1X^*}(Y_i)$ for some elevation $Y_i \rightarrow \hat X$ of some $Y_i \rightarrow X$ and some $g \in \pi_1X^*$. 
Thus, $HgStab_{\pi_1X^*}(Y_i)=gStab_{\pi_1X^*}(Y_i)$, so $g^{-1}HgStab_{\pi_1X^*}(Y_i)=Stab_{\pi_1X^*}(Y_i)$ and $H$ is conjugate into $Stab_{\pi_1X^*}(Y_i)$. 
\end{proof}

\begin{lemma}\label{lem:contfix}
If a non-trivial subgroup $H <\pi_1X^*$ fixes a point in $\bar X^*$, then that point is unique, and in particular the fixed-point space $(\bar X^*)^H$ is contractible.
\end{lemma}

Before proving the lemma, we need an auxiliary result.

\begin{proposition}\label{prop:boundedpiece}
Let $X^*=\langle X | \{Y_i\}^k_{i=1} \rangle$ be a cubical presentation that satisfies the $C(n)$ condition for some $n \geq 2$, and where each $Y_i$ is compact. Then there is a uniform upper-bound $\mathcal{L} \geq 0$ on the size of pieces in $\widetilde{X^*}$.
\end{proposition} 

\begin{proof}
Since each $Y_i$ is compact, $\pi_1Y_i$ is finitely generated for each $i \in \{1, \ldots, k\}$, and since the cubical presentation has finitely many relations $Y_i \rightarrow X$, then finitely many elements of $\pi_1X$ suffice to generate all  the $\pi_1Y_i$'s. The compactness of the $Y_i$'s implies additionally that there is an uniform upper-bound $\ell_i$ for the length of minimal-length paths representing the generators of $\pi_1Y_i$, thus, there is a bound $\mathcal{L}= \max\{\ell_i\}$ for the size of pieces arising in generators of $\pi_1Y_1, \ldots, \pi_1Y_k$. Since $X^*$ satisfies the $C(n)$ condition for $n \geq 2$, then each path representing a generator is a concatenation of at least $2$ pieces, and in particular any such piece $p$  satisfies $|p|< \ell_i \leq \mathcal{L}$ whenever it arises in a generator of a $\pi_1Y_i$.

Assume that there is an essential closed path $\sigma \rightarrow Y_i$  that is a concatenation of pieces having a piece $p_r$ with $|p_r| > \mathcal{L}$. By the discussion above, we can assume in particular that $\sigma$ does not represents a generator of a $\pi_1Y_{i'}$ for any $i' \in \{1, \ldots, k\}$.
Note that any expression for $\sigma$ as a concatenation of pieces can be further expressed as a concatenation of pieces in generators of $\pi_1Y_i$. Write $\sigma=\alpha_1 \cdots \alpha_m=p_1 \cdots p_r   \cdots p_n$, where each $\alpha_{j}$ is a (not necessarily distinct) generator of $\pi_1Y_i$. Since $|p_r| > \mathcal{L}$, then there is an $\alpha_{j}$ with $j \in \{1, \ldots, m\}$ for which  $p_r \cap \alpha_{j}= \alpha_{j}$. Thus, $\alpha_{j}$ can be expressed as a single piece $q_r \subset p_r$, contradicting the discussion in the previous paragraph, and hence the $C(n)$ condition.
\end{proof}

The proof of Lemma~\ref{lem:contfix} now combines Proposition~\ref{prop:boundedpiece} with the hypothesised symmetry of the cubical presentation under consideration.

\begin{proof}[Proof of Lemma~\ref{lem:contfix}]
Let $H<\pi_1X^*$. Assume $\xi \neq \xi'$ are cone-vertices of $\bar X^*$ fixed by $H$, so $H\xi=\xi$ and $H\xi'=\xi'$. Then Lemma~\ref{lem:iffix} implies that $H \subset g^{-1}Stab_{\pi_1X^*}(Y_i)g \cap g'^{-1} Stab_{\pi_1X^*}(Y_{i'})g'$ for some $i, i' \in  \{1, \ldots, k\}$ and $g, g' \in \pi_1 X^*$. We claim that the collection $\{Stab_{\pi_1X^*}(Y_i)\}$ is malnormal, and thus $H$ is the trivial subgroup. To this end, we first show that the collection $\{Stab_{\pi_1X}(\widetilde Y_i)\}$ is malnormal in $\pi_1X$.

Suppose that the intersection $Stab_{\pi_1X}(\widetilde Y_i)^{\tilde g} \cap Stab_{\pi_1X}(\widetilde Y_j)^{\tilde g'}$ is infinite for some $i, i' \in  \{1, \ldots, k\}$ and $\tilde g, \tilde g' \in \pi_1 X$. Since $X^*$ is symmetric, then $[Stab_{\pi_1X}(\widetilde Y_i):\pi_1Y_i]< \infty$ and $[Stab_{\pi_1X}(\widetilde Y_j):\pi_1Y_j]< \infty$, so $\pi_1 Y_i^{\tilde g} \cap \pi_1 Y_j^{\tilde g'}$ has finite index in $Stab_{\pi_1X}(\widetilde Y_i)^{\tilde g} \cap Stab_{\pi_1X}(\widetilde Y_j)^{\tilde g'}$. Thus, $\pi_1 Y_i^{\tilde g} \cap \pi_1 Y_j^{\tilde g'}$ is infinite, and in particular contains an infinite order element $\tilde h$. Therefore, the axis of $\tilde h$ in $\widetilde X$ is an unbounded piece between elevations $\widetilde Y_i$ and $\widetilde Y_j$ of $Y_i$ and $Y_j$, contradicting Proposition~\ref{prop:boundedpiece}, and thus the $C(9)$ condition. 
We conclude that $Stab_{\pi_1X}(\widetilde Y_i)^{\tilde g} \cap Stab_{\pi_1X}(\widetilde Y_j)^{\tilde g'}$ is finite; since $\pi_1 X$ is torsion-free, then in fact $Stab_{\pi_1X}(\widetilde Y_i)^{\tilde g} \cap Stab_{\pi_1X}(\widetilde Y_j)^{\tilde g'}$ must be trivial, so the collection $\{Stab_{\pi_1X}(\widetilde Y_i)\}$ is malnormal in $\pi_1X$. 

To promote this to malnormality of the collection $\{Stab_{\pi_1X^*}(Y_i)\}$ in $\pi_1X^*$, let $h \in Stab_{\pi_1X^*}( Y_i)^g \cap Stab_{\pi_1X^*}( Y_j)^{g'}$, and let $\tilde{h_i} \in Stab_{\pi_1X}(\widetilde Y_i)^{\tilde g}$ and $\tilde{h_j} \in Stab_{\pi_1X}(\widetilde Y_j)^{\tilde g'}$ be preimages of $h$, so $\tilde{h_i}\tilde{h_j}^{-1}$ is an element of $\pi_1 \hat X$.
Since $\tilde{h_i} \in Stab_{\pi_1X}(\widetilde Y_i)^{\tilde g}$ and $\tilde{h_j} \in Stab_{\pi_1X}(\widetilde Y_j)^{\tilde g'}$ and as mentioned before, by symmetry,  $[Stab_{\pi_1X}(\widetilde Y_i):\pi_1Y_i]< \infty$ and $[Stab_{\pi_1X}(\widetilde Y_j):\pi_1Y_j]< \infty$, then there are paths $\sigma_i \rightarrow Y_i, \sigma_j \rightarrow Y_j$ such that for some $k,k' \in \naturals$, the paths $\sigma_i^k, \sigma_j^{k'}$ bound disc diagrams $D_i \rightarrow Y_i,D_j \rightarrow Y_j$ in $\widetilde {X^*}$, each of which can be assumed to have minimal complexity, and therefore consisting of a single cone-cell. 
Moreover, by the discussion above, $\sigma_i$ and $\sigma_j$ represent lifts of conjugates of $\tilde{h_i}$ and $\tilde{h_j}$, so there are paths $\sigma'_i, \sigma'_j$  such that $\sigma'_i\sigma_j^{'-1}$ is a closed path that represents $\tilde{h_i}\tilde{h_j}^{-1}$. Thus, $\sigma'_i\sigma_j^{'-1}$ bounds a disc diagram $D$ in $\widetilde {X^*}$, and there are paths $\ell_i, \ell_i', \ell_j, \ell_j'$ that are lifts of the corresponding conjugating elements, so that $\sigma_i\ell_i\sigma'_i\ell'_i$ and $\sigma_j\ell_j\sigma'_j\ell'_j$ bound square-disc diagrams $R_i, R_j$ in $\widetilde {X^*}$. 
We can further assume that   $R_i,R_j$ and $D$ are chosen to have minimal area and minimal complexity amongst all possible diagrams with the same corresponding boundary.  For $R_i$ and $R_j$, this means in particular that there are no cornsquares on any of $\sigma_i, \sigma'_i, \sigma_j, \sigma'_j$, so each of these paths determines a hyperplane carrier in the corresponding diagram $R_i$ or $R_j$, and more generally that $R_i$ and $R_j$ are \emph{grids}, in the sense that they are cubically isomorphic to a product of combinatorial intervals $I_n \times I_m$ for suitable $m,n \in \naturals$. 
We note that, of course, $R_i$ or $R_j$ could be degenerate square diagrams, in which case $\sigma_i=\sigma'_i$ or $\sigma_j= \sigma'_j$, and the argument in the next paragraph becomes a bit simpler.

If $D$ is a square disc diagram, then $\sigma'_1$ and $\sigma'_j$ are homotopic rel their endpoints in $\widetilde{X^*}$, so  $\tilde{h_i}=\tilde{h_j}$, contradicting malnormality in $\pi_1X$. We claim that this is the only possibility for $D$.
Assume towards a contradiction that $D$  contains at least one cone-cell, and observe that since any corners or spurs on $\partial D$ could be removed, improving the choices made in the previous paragraph, that cone-cell must be a shell $S$. Consider the disc diagram $F=D_i\cup_{\sigma_i}R_i \cup_{\sigma'_i}D\cup_{\sigma'_j}R_j\cup_{\sigma_j}D_j$. 
Then, since $S$ is a shell of $D$, its innerpath is a concatenation of at most $4$ pieces, and either the outerpath of $S$  is a subpath of $\sigma'_i \subset R_i$, or a subpath of $\sigma'_j \subset  R_j$, or a subpath of their concatenation. In the first case, since $\partial S\cap R_i$ in $F$ is the outerpath of $S$ in $D$, then $\partial S\cap R_i$ is a single piece (a wall piece if $R_i$ is not a degenerate diagram, and a cone piece otherwise), so $\partial S$ is the concatenation of at most $5$ pieces, contradicting the $C(9)$ condition. Similarly if the outerpath of $S$ is a subpath of $\sigma'_j$. 
In the last case, $\partial S$ is the concatenation of at most $6$ pieces: the $4$ pieces coming from its innerpath and a piece coming from each of $\sigma'_i, \sigma'_j$, contradicting again the $C(9)$ condition. 

So $\{Stab_{\pi_1X^*}(Y_i)\}$ is malnormal in $\pi_1X^*$, and either $H=\{1\}$, or $\xi=\xi'$.
\end{proof}

\begin{remark} In the proof above, we use that $\pi_1 X$ is torsion-free, which follows trivially from the assumption that $X$ is finite-dimensional, since $X$ is aspherical, but is true also for infinite dimensional non-positively curved cube complexes, and is a consequence of CAT(0) geometry~\cite{BridsonHaefliger, Leary_KanThurston}.
\end{remark}
  
We can finally conclude that, in fact, the finite subgroups of $\pi_1X^*$ have non-empty fixed-point sets under their action on $\bar X^*$:  
  
\begin{corollary}\label{cor:fix}
Every finite subgroup $H <\pi_1X^*$ fixes a point in $\bar X^*$.
\end{corollary}

\begin{proof}
Let $H$ be a finite subgroup of  $\pi_1X^*$. If $H$ is trivial, then it fixes all of $\bar X^*$. Otherwise, since the action of $H$ on $\bar X^*$ cannot be free by Lemma~\ref{lem:nonfree}, for any non-trivial $h \in H$, there exists $k \in \naturals$ for which $h^k$ fixes a point in $\bar X^*$. This implies that each cyclic subgroup $\langle h \rangle$ fixes a point, and in particular that every $h \in H$ fixes a point in $\bar X^*$. Indeed, if $h^kv=v$ for some $k \in \naturals$ and $v \in \bar X^*$, then $hv=h(h^kv)=h^k(hv)$, so $h^k$ fixes $hv$ and by Lemma~\ref{lem:contfix} $hv=v$.

To see now that every element of $H$ fixes the same point, we use Burnside's Lemma: since $H$ is a finite group acting on a finite set $S$ (the union of the orbits of the fixed points of non-trivial elements), and each non-trivial element fixes a unique point, if $m$ is the number of orbits of elements, then $ m= \frac{|S|+|H|-1}{|H|}$, so $|S|=(m-1)|H|+1$ and in particular, as the size of each orbit must divide $|H|$, at most one orbit of points can have size $<|H|$. There must therefore be $m-1$ orbits of size $|H|$, and a single orbit of size $1$, yielding a fixed point in $\bar X^*$.
\end{proof}

Putting together the previous results, we obtain:

\begin{theorem}\label{thm:B}
Let $X^*=\langle X | \{Y_i\}^k_{i=1} \rangle$ be a symmetric cubical presentation that satisfies the $C(9)$ condition. Let $\pi_1 X^*=G$. If  $dim(X)\leq 2$ and each $Y_i$ is homotopy equivalent to a graph, then there is a quotient $\bar X^*$ of $\widetilde{X^*}$ that is an $\underbar EG$, so $cd_\rationals (G) \leq dim(\bar X^*) \leq 2$. 
\end{theorem}

\begin{proof}
This follows from Corollary~\ref{cor:lowdim}, Lemma~\ref{lem:contfix}, and Corollary~\ref{cor:fix}.
\end{proof}

We can go a step further if we assume that $X$ has a finite regular cover where every $Y_i \rightarrow X$ lifts to an embedding, as is observed in~\cite[4.4]{WiseIsraelHierarchy} for $C'(\frac{1}{20})$ cubical presentations.

\begin{lemma}\label{lem:vcd} Let $X^*=\langle X | \{Y_i\}_{i \in I} \rangle$ be a symmetric $C(9)$ cubical presentation where $dim(X)\leq 2$ and each $Y_i$ is homotopy equivalent to a graph. If there exists a finite regular cover $\overbow{X} \rightarrow X$  where  $Y_i \rightarrow X$ lifts to an embedding for each $i \in I$, then $\pi_1X^*$ is virtually torsion-free.
\end{lemma}

\begin{proof}
Consider the covering $\overbow{X^*} \rightarrow X^*$ induced by $\overbow{X} \rightarrow X$, which is a finite-degree covering because  $\overbow{X} \rightarrow X$ is. If $g \in \pi_1X^* - \{1\}$ satisfies $g^n=1$ for some $n \in \naturals$, then Corollary~\ref{cor:fix} implies that if $\sigma \rightarrow X^*$ is a closed path representing $g$, then $\sigma^n$ is conjugate to a closed path in some $Y_i$. Now since $g$ is non-trivial, $\sigma$ does not lift to a closed path in $Y_i$. Since  $Y_i \rightarrow X$ lifts to an embedding  $Y_i \rightarrow \overbow{X}$, then $\sigma$ cannot lift to a closed path in $\overbow{X}$ either, so $\sigma$ is not closed in  $\overbow{X^*}$, and $g \notin \pi_1\overbow{X^*}$.
\end{proof}

If a group $G$ has finite virtual cohomological dimension, then $vcd(G)\leq dim (\underbar EG)$, since a torsion-free subgroup $G' <G$ of finite index has $cd(G)=vcd(G)$, and $dim (\underbar EG)$ gives an upper bound for the cohomological dimension of any torsion-free subgroup of $G$. In particular:

\begin{corollary}\label{cor:vcd}
Let $X^*=\langle X | \{Y_i\}_{i \in I} \rangle$ be a symmetric $C(9)$ cubical presentation with $G=\pi_1X^*$. If $dim(X)\leq 2$ and each $Y_i$ is homotopy equivalent to a graph, and there exists a finite regular cover $\overbow{X} \rightarrow X$  where  $Y_i \rightarrow X$ lifts to an embedding for each $i \in I$, then $vcd(G)\leq 2$. \qed
\end{corollary}

\section{Examples}\label{sec:ex}

Using Theorems~\ref{thm:A} and~\ref{thm:B}, we can provide classifying spaces, or classifying spaces for proper actions, in the situations outlined below:

\begin{example}\label{const:covers}(Forcing small-cancellation by taking covers)
If $X$ is a compact non-positively curved cube complex with hyperbolic fundamental group, and $H_1,...,H_k$ are quasiconvex subgroups of $\pi_1X$ that form a malnormal collection, then for each $n>0$ there are finite index subgroups $H'_i \subset H_i$ and local isometries $Y_i \rightarrow X$ with $Y_i$ compact and $\pi_1Y_i=H'_i$ such
that $X^*= \langle X | \{Y_i\} \rangle$ satisfies $C(n)$. How one achieves this condition is explained in~\cite[3.51]{WiseIsraelHierarchy}. To guarantee that $X^*$ is symmetric, it suffices to require that $H_i=C_{\pi_1X}(H_i)$. When $n\geq9$, $dim(X)\leq 2$ and $H_1,...,H_k$ are free, we conclude that $\bar X^*$ is a classifying space for proper actions for $\pi_1X^*$, and in particular that $\pi_1X^*$ has rational cohomological dimension $cd_\rationals (\pi_1X^*) \leq 2$.
\end{example}

\begin{example}[Forcing small-cancellation by creating noise]\label{ex:noise}
Let $X$ be a compact non-positively curved cube complex with non-elementary hyperbolic fundamental group, as above. Then by~\cite[3.2]{Arenas2023}, for all $k \geq 1$ there exist (infinitely many choices of) free non-abelian subgroups $\{H_1, \ldots, H_k\}$, and cyclic subgroups $z_i <H_i$  such that there exist local isometries $Y_i \rightarrow X$ with $Y_i$ compact and $\pi_1Y_i=z_i$, so that the cubical presentation $X^*= \langle X | \{Y_i\} \rangle$ is minimal, and satisfies the $C(9)$ condition. Thus, if $\dim(X)\leq 2$,  then $X^*$ is a classifying space for all such examples.
\end{example}

 We now move away slightly from the hyperbolic setting with a construction that is described in~\cite[5.r]{WiseIsraelHierarchy}.

\begin{example}[Some 2-dimensional Artin groups]\label{ex:artins}
Let $\Gamma$ be a finite simplicial graph, and let 
\begin{equation}\label{eq:artindef} A_\Gamma=\langle v_i\in V(\Gamma) | (v_i,v_j)^{m_{ij}}=(v_i,v_j)^{m_{ij}} \text{ if } i<j \rangle
\end{equation}
be an Artin group on $\Gamma$, where either $m_{ij}$ is an integer $\geq 2$ or $m_{ij}=\infty$,  the notation $(v_i,v_j)^{m_{ij}}$ denotes the first half of the word $(v_iv_j)^{m_{ij}}$, and we adopt the (standard) convention that $(v_iv_j)^\infty$ signifies that there is no relation between $v_i$ and $v_j$. 

Let $B_\Gamma$ denote the bouquet of $|V|$ circles, where each circle is labelled by a $v_i \in V$. Consider the cubical presentation \begin{equation}\label{eq:artin1}
B^*_\Gamma=\langle B_\Gamma | Y_{ij} \rightarrow B_\Gamma \rangle
\end{equation} where, for each $m_{ij} < \infty$, the complex $Y_{ij}$ is the Cayley graph of the 2-generator Artin group $\langle v_i,v_j | (v_i,v_j)^{m_{ij}}=(v_j,v_i)^{m_{ij}} \rangle$ and $Y_{ij} \rightarrow B_\Gamma$ is the obvious covering projection. Then $\pi_1 B^*_\Gamma=A_\Gamma$, so this is a ``non-trivial'' cubical presentation for $A_\Gamma$.

Note that $B^*_\Gamma$ is a minimal cubical presentation, and that the $C(9)$ condition is satisfied provided that $5\leq m_{ij}\leq \infty$ for all $i <j$. 
Indeed, since $B^*_\Gamma$ is a bouquet of circles, and hence 1-dimensional, there are no non-trivial wall-pieces in $B^*_\Gamma$. 
The cone-pieces correspond to intersections between distinct elevations of the $Y_{ij}$'s, and for a pair of relations   $Y_{ij}, Y_{i'j'}$, these intersections are (subpaths of) bi-infinite lines of the form $v^\infty_\ell$ where one of $i,j$ and one of $i'j'$ are equal to $\ell$. Since an essential cycle in a $Y_{ij}$ lifts to a path that intersects $v^\infty_\ell$ in a single edge, and the girth of any such cycle is $\geq 5$, then no essential cycle is the concatenation of $<10$ pieces.

It is a well-established result~\cite{CharneyDavis95} that $2$-dimensional Artin groups are exactly those where no triangle in the defining graph has labels $(2,3,4),(2,2,n),(2,3,3)$ or $(2,3,5)$. This class contains in particular the class of \emph{extra large type Artin groups}, which are precisely those with labels $4\leq m_{ij}\leq \infty$ for all $i <j$.  In light of the discussion above, we can partially recover the result that extra large type Artin groups are 2-dimensional as a corollary of Theorem~\ref{thm:A}:

\begin{theorem}\label{thm:exampleartins2}
Let $A_\Gamma$ be an Artin group on $\Gamma$ given by a presentation~\eqref{eq:artindef}, where  $5\leq m_{ij}\leq \infty$ for all $i <j$. Then $B^*_\Gamma$ is a $K(A_\Gamma,1)$, and in particular, $A_\Gamma$ is torsion-free and $cd(A_\Gamma)\leq 2$. 
\end{theorem}

We note that, while the cubical presentation~\eqref{eq:artin1} can be viewed as a graphical presentation, asphericity of $B_\Gamma^*$ cannot be deduced from the asphericity results in~\cite{Gromov2003, Ollivier06, Gruber15}, since these handle only graphical presentations where the relators are cycles, rather than arbitrary --possibly non-compact-- graphs. While we state Theorem~\ref{thm:exampleartins2} for Artin groups whose associated cubical presentations satisfy the cubical $C(9)$ condition, we believe that the result can be recovered for \emph{all}  Artin groups where  $3\leq m_{ij}\leq \infty$ for all $i <j$, since in that case the $C(6)$ condition is satisfied, and in the graphical case this condition should be enough to recover a suitable form of Theorem~\ref{thm:tric}, and thus of Theorem~\ref{thm:A}. We have not, however, explored this  possibility closely. 

It is claimed in~\cite[5.70]{WiseIsraelHierarchy} that the cubical presentation~\eqref{eq:artin1} satisfies the $C(6)$ condition whenever the defining graph has no triangles with labels of the form $(2,3,4),(2,2,n),(2,3,3)$ or $(2,3,5)$, but no justification is presented in that text. Nevertheless, if the claim is true, then all $2$-dimensional Artin groups admit minimal $C(6)$ graphical presentations.
\end{example}

When $X$ is a square complex whose fundamental group is not hyperbolic, general constructions of $C(9)$ cubical presentations are harder to come by, but can be produced by hand in specific situations. For instance, if $X_\Gamma$ is the Salvetti complex of a RAAG whose defining graph $\Gamma$ has no triangles, then $dim(X_\Gamma) \leq 2$, and one can produce a wealth of cubical presentations $X^*_\Gamma=\langle X_\Gamma | \{Y_i\} \rangle$ that satisfy the $C(9)$ condition and are minimal. This is outlined in~\cite[3.s (4)]{WiseIsraelHierarchy}, and will be elaborated upon in forthcoming work~\cite{Arenas2023pi}.

 Thus, Theorems~\ref{thm:A} and~\ref{thm:B} are widely applicable, even when one starts with a non-positively curved cube complex $X$ whose fundamental group is far from being hyperbolic.

\bibliographystyle{alpha}
\bibliography{bib9.bib}
\end{document}